\documentclass[12pt]{article}
\usepackage{epsf}
\usepackage{amsbsy,amsmath}
\usepackage{mathtools}
\usepackage{mathrsfs}
\usepackage{amsfonts}
\usepackage{amssymb}
\usepackage{enumitem}
\usepackage{eucal}
\usepackage{graphics,mathrsfs}
\usepackage{amsthm}
\usepackage{secdot}
\usepackage{authblk}
\newtheorem{theorem}{Theorem}[section]
\newtheorem{lemma}[theorem]{Lemma}

\theoremstyle{definition}
\newtheorem{definition}[theorem]{Definition}

\newtheorem{remark}[theorem]{Remark}
\numberwithin{equation}{section}

\numberwithin{equation}{section}

\newsavebox{\savepar}

\begin{document}
	
	\title{Least energy sign-changing solution of fractional $p$-Laplacian problems involving singularities}
	\author{\small Sekhar Ghosh\footnote{sekharghosh1234@gmail.com}$^{~,1}$, Kamel Saoudi\footnote{kmsaoudi@iau.edu.sa (Corresponding author)},$^{\,,2}$ Mouna Kratou\footnote{mmkratou@iau.edu.sa},$^{\,,2}$ \& Debajyoti Choudhuri\footnote{dc.iit12@gmail.com}$^{~,1}$\\
		\small{$^1$\it Department of Mathematics, National Institute of Technology Rourkela, India}\\
		\small{$^2$\it Basic and Applied Scientifc Research Center, Imam Abdulrahman Bin Faisal University,,}\\ \small{\it P.O. Box 1982, 31441, Dammam, Saudi Arabia}}
	\date{}
	\maketitle
	\vspace{-1cm}
	\begin{abstract}
		\noindent In this paper we study the existence of a least energy sign-changing solution to a nonlocal elliptic PDE involving singularity by using the Nehari manifold	method, the constraint variational method and Brouwer degree theory.\\
		{\bf keywords}: Sign-changing solutions; Fractional $p$-Laplacian;
		Nehari manifold.
	\end{abstract}

	\section{Introduction and Main results}\label{intro}
	
	In this paper we consider the following fractional $p$-Laplacian problem involving singularity and a power nonlinearity.
	\begin{align}\label{1.1}
	(P)~~~~~~~~~
	\left\{ \begin{aligned} (-\Delta_p)^\alpha u&= \lambda g(u)+
	f(x,u)~\mbox{in}~\Omega,\\
	u&= 0 ~\mbox{in}~\mathbb{R}^N\setminus \Omega,
	\end{aligned}\right.
	\end{align}
	where, $\Omega\subset\mathbb{R}^{N}$ is a bounded domain with smooth
	boundary $\partial \Omega$, $\lambda>0$, $p\in(1,+\infty)$, $\alpha\in(0,1)$,
	$N>p\alpha $, $0<\delta<1,$ $f:\overline{\Omega}\times\mathbb{R}\to \mathbb{R}$ is
	continuous,  $g:\mathbb{R}^+\to\mathbb{R}^+$ is continuous, nonincreasing on $(0,+\infty)$ such that $c_1=\liminf_{t\to 0^+}g(t)t^{\delta}\leq\limsup_{t\to 0^+}g(t)t^{\delta}=c_2$, 
	for some $c_1,c_2>0$ and the fractional $p$-Laplacian operator, $(-\Delta)_p^\alpha$ is defined as,
	\begin{equation*}\label{operator}
	(-\Delta_p)^\alpha
	u(x)=C_{N,\alpha}P.V.\int_{\mathbb{R}^N}\frac{|u(x)-u(y)|^{p-2}(u(x)-u(y))}{|x-y|^{N+p\alpha
	}}dy,~~ x\in \mathbb{R}^N,
	\end{equation*}
	where $C_{N,\alpha}$ is a normalizing constant.
	
	\noindent One of the classical topic in the analysis of PDEs is the study of existence and multiplicity of nonnegative solutions for both the $p$-Laplacian and the fractional $p$-Laplacian operator involving concave-convex nonlinearity and singularity-power nonlinearity. In the recent past there has been considerable interest in studying the following general fractional $p$-Laplacian problem involving singularity.
	\begin{eqnarray}\label{refer nonlocal}
	(-\Delta_p)^s u&=& \frac{\lambda a(x)}{u^\gamma}+Mf(x,u)~\text{in}~\Omega,\nonumber\\
	u&=&0~\text{in}~\mathbb{R}^N\setminus\Omega,\\
	u&>& 0~\text{in}~\Omega\nonumber,
	\end{eqnarray}
	where $N>ps$, $M\geq0$, $a:\Omega\rightarrow\mathbb{R}$ is a nonnegative bounded function.
	Ghanmi \& Saoudi \cite{ghanmi2016multiplicity} guaranteed the existence of multiple weak solutions to the problem (\ref{refer nonlocal}), for $0<\gamma<1$ and $1<p-1<q\leq p_s^*$ by using the Nehari manifold method. Recently, multiplicity and H\'{o}lder regularity of solutions to the problem \eqref{refer nonlocal} has been studied by Saoudi et al. \cite{saoudi2018multiplicity}. On the other hand, for $p = 2$, the problems of the type \eqref{refer nonlocal}, have been investigated by many
	researchers. For references see \cite{mukherjee2016dirichlet, saoudi2017critical, saoudi2018multiplicity} and the references therein.\\ 
	The existence of a sign-changing solution of nonlinear elliptic PDEs with power nonlinearities has been studied extensively for the $p$-Laplacian operator as well as the fractional $p$-Laplacian operator. We refer the reader to see \cite{AS,BW,BSRT,chang, GTZ, LLW2, WZ1, Wei} and the references therein. Consider the nonlocal problem 
	\begin{eqnarray}\label{1.3} \left\{
	\begin{array}{ll}
	(-\Delta)_p^\alpha u= f(x,u)~~& \mbox{in}~ \Omega,\\
	u=0~~&\mbox{on}~\mathbb{R}^N\setminus\partial \Omega,
	\end{array}
	\right.
	\end{eqnarray}
	For $p=2$, the authors in \cite{CW}, has studied the problem \eqref{1.3}, 
	where the fractional Laplacian operator is defined through spectral decomposition to obtain the sign-changing solution. The method of harmonic extension was introduced by Caffarelli and Silvestre \cite{CaSi} to
	transform the nonlocal problem in $\Omega$ to a local problem in the
	half cylinder $\Omega\times (0, +\infty)$, by using an equivalent definition of the fraction Laplacian operator \cite{BCde}.\\
	For $p\in(1, \infty)$, the problem studied by Chang et al. \cite{chang}, where the authors have guaranteed the existence of a sign-changing solutions by using Nehari manifold method.
	Recently, the study of the nonlocal problems with singularity has drawn interest to many researchers. For recent studies on nonlocal PDEs involving singularities, we refer \cite{crandall1977dirichlet, dhanya2012global, giacomoni2009multiplicity, giacomoni2007multiplicity, haitao2003multiplicity, hirano2004existence, saoudi2018multiplicity} and the references therein.\\
	The main goal of this article is to obtain a sign-changing
	solutions to the nonlocal problem \eqref{1.1} involving singularity. For $p\neq 2$, the harmonic extension method can non be applied on an equivalent definition of$(-\Delta)_p^\alpha$. On a similar note, we can not have the decomposition $\Phi(u)=\Phi(u^+)+\Phi(u^-)$ for $u=u^++u^-$, where $\Phi$ is the corresponding energy functional to the problem \eqref{1.1}.
	Therefore, by using the method as in \cite{CW}, one can not guarantee the existence of a sign-changing solution.\\
	Therefore, we will apply the Nehari manifold method combining with a constrained variational method and Brouwer degree theory to obtain a least energy sign-changing solution.\\
	We first recall some preliminary results on the fractional Sobolev space \cite{Ad, DPV}. Let $\Omega\subset\mathbb{R}^N$ is a bounded domain with smooth boundary and $\alpha\in(0,1)$.
	We denote the fractional Sobolev space by $W^{\alpha,p}(\Omega)$ equipped with the norm
	\begin{equation*}
	\|u\|_{W^{\alpha, p}(\Omega)}=\|u\|_{L^p(\Omega)}+(\int_{\Omega\times\Omega}\frac{|u(x)-u(y)|^p}{|x-y|^{N+\alpha p}}dxdy)^{\frac{1}{p}}.
	\end{equation*}
	We set, $Q=\mathbb{R}^{2N}\setminus((\mathbb{R}^N\setminus\Omega)\times(\mathbb{R}^N\setminus\Omega))$, then the space $(X, \|..\|_X)$ is defined by
	\begin{eqnarray}
	X&=&\left\{u:\mathbb{R^N}\rightarrow\mathbb{R}~\text{is measurable}, u|_{\Omega}\in L^p(\Omega) ~\text{and}~\frac{|u(x)-u(y)|}{|x-y|^{\frac{N+p\alpha}{p}}}\in L^{p}(Q)\right\}\nonumber
	\end{eqnarray}
	equipped with the Gagliardo norm 
	\begin{eqnarray}
	\|u\|_X&=&\|u\|_{p}+\left(\int_{Q}\frac{|u(x)-u(y)|^p}{|x-y|^{N+p\alpha}}dxdy\right)^{\frac{1}{p}}.\nonumber
	\end{eqnarray}
	Here $\|u\|_{p}$ refers to the $L^p$-norm of $u$. We then define the space 
	\begin{eqnarray}
	X_0&=&\left\{u\in X: u=0 ~\text{a.e. in}~ \mathbb{R}^N\setminus\Omega\right\}\nonumber
	\end{eqnarray}
	equipped with the norm
	\begin{eqnarray}
	\|u\|&=&\left(\int_{Q}\frac{|u(x)-u(y)|^p}{|x-y|^{N+p\alpha}}dxdy\right)^{\frac{1}{p}}.\nonumber
	\end{eqnarray}
	\noindent The best Sobolev constant is defined as 
	\begin{equation}\label{sobolev const}
	S=\underset{u\in X_0\setminus\{0\}}{\inf}\cfrac{\int_{Q}\cfrac{|u(x)-u(y)|^p}{|x-y|^{N+p\alpha}}dxdy}{\left(\int_\Omega|u|^{p_{\alpha}^*}dx\right)^{\frac{p}{p_{\alpha}^*}}}
	\end{equation}
	For $p > 1$, the space $X_0$ is a uniformly convex Banach space \cite{servadei2012mountain, servadei2013variational} and the embedding $X_0 \hookrightarrow L^{q}(\Omega) $ is compact for $q\in[1, p_{\alpha}^*)$ and is continuous for $q\in[1, p_{\alpha}^*]$, where $p_{\alpha}^*$ is the Sobolev conjugate of $p$, defined as $p_{\alpha}^*=\frac{Np}{N-p\alpha}$.\\ 
	Henceforth, we have the following assumptions on $f$ and $g$.
		\begin{description}
			\item[($f_1$)]$f\in C(\overline{\Omega}\times \mathbb{R})$, $\lim\limits_{|u|\to 0}\frac{f(x,u)}{|u|^{p-2}u}=0$, uniformly in $x$;
			\item[($f_2$)]there exist constants $C_0>0$ and $q\in (p, p_{\alpha}^*)$ with $p_{\alpha}^*=\frac{pN}{N-p\alpha}$ such that
			\begin{equation*}
				|f(x,u)|\le C_0(1+|u|^{q-1}),~ \forall
				u\in \mathbb{R}, \forall x\in \Omega;
			\end{equation*}
			\item[($f_3$)]there exist $\mu>p$ and $M_0>0$ such that $f(x,u)u\ge \mu F(x,u)>0$ for $|u|\ge M_0$, uniformly in
			$x$, where $F(x,u)\doteq \int_0^uf(x,\tau)d\tau$;
			\item[($f_4$)]$\lim\limits_{|u|\rightarrow +\infty}\frac{f(x,u)}{|u|^{p-2}u}=+\infty$ uniformly in $x$;
			\item[($f_5$)]$\frac{f(x,u)}{|u|^{p-2}u}$ is strictly increasing on $(0,+\infty)$ and strictly decreasing on $(-\infty,0)$, uniformly in $x$.
			\item[($g_1$)] $g:\,\mathbb{R}\setminus\{0\}\to\mathbb{R}^+$ continuous on $\mathbb{R}\setminus\{0\}$,  $g$ is nondecreasing on $(-\infty, 0)$ and $g$ is nonincreasing on $(0,+\infty)$,
			\item[($g_2$)] $c_1\leq \liminf_{t\to 0^+}g(t)t^{\delta}\leq\limsup_{t\to 0^+}g(t)t^{\delta}=c_2$ for some $c_1,c_2>0$ and\\ $d_1\leq \liminf_{t\to 0^-}g(t)t^{\delta}\leq\limsup_{t\to 0^-}g(t)t^{\delta}=d_2$ for some $d_1,d_2>0$.
		\end{description}
		\begin{remark}
			\begin{enumerate}
				\item By $(f_5)$ it follows that 
				\begin{equation*}
					t^2f'(t)-(p-1)f(t)t>0, \forall |t|>0.
				\end{equation*}
				\item From { $(g_2)$}, $g$ is singular at the origin and $\lim_{t\to 0^{\pm}}g(t)=\pm\infty$.
			\end{enumerate}
		\end{remark}
We now define a weak solution to the problem defined in (\ref{1.1}).
\begin{definition}\label{defn weak}
	A function $u\in X_0$ is a weak solution to the problem (\ref{1.1}), if 
	\begin{align*}
 &\int_{Q}\frac{|u(x)-u(y)|^{p-2}(u(x)-u(y))(\phi(x)-\phi(y))}{|x-y|^{N+p\alpha}}dxdy-\lambda\int_{\Omega}g(u)\phi-\int_{\Omega}f(x, u)\phi=0
	\end{align*}
\end{definition}
\noindent for each $\phi\in X_0.$ The corresponding Euler-Lagrange energy functional is
\begin{equation*}
I_{\lambda}(u)=\frac{1}{p}\int_{\Omega}\frac{|u(x)-u(y)|^{p}}{|x-y|^{N+p\alpha}}dxdy-\lambda\int_{\Omega}G_\lambda(u)dx-\int_{\Omega}F_\lambda(x, u)dx.
\end{equation*}
 It is easy to observe that $I_{\lambda}$ is not $C^1$ due the presence of the singular term in it but $I_{\lambda}$ is continuous and G\^{a}teaux differentiable (see Corollary 6.3 of \cite{saoudi2018multiplicity}). Therefore, we can not apply Nehari manifold method corresponding to the functional $I_{\lambda}$. Hence, we will establish the existence of a sign-changing solution to the problem \eqref{1.1} by obtaining a critical point to a $C^1$ cutoff functional. We define,
\begin{equation*}
	\Lambda=\inf\{\lambda>0:~\text{The problem (\ref{1.1}) has no weak solution}\}.
\end{equation*}
We now state the existence of a unique solution due to \cite{canino2017nonlocal} to the following problem.
\begin{align}\label{squasinna}
(-\Delta_p)^s w&=\lambda g(w)~\text{in}~\Omega,\nonumber\\
w&>0~\text{in}~\Omega,\nonumber\\
w&=0~\text{in}~\mathbb{R}^N\setminus\Omega.
\end{align}
\begin{lemma}
	Assume $0<\delta<1$ and $\lambda>0$. Then the problem \eqref{squasinna} has a unique solution, $\underline{u}_{\lambda}\in W_0^{s, p}(\Omega)$, such that for every $K\subset\subset\Omega$, $ess.\underset{K}{\inf}\,\underline{u}_\lambda>0$.
\end{lemma}
\noindent Define,
\[   
\overline{g}(t) = 
\begin{cases}
g(t), &~\text{if}~t>\underline{u}_{\lambda}\\
g(\underline{u}_{\lambda}),&~\text{if}~t\leq \underline{u}_{\lambda}
\end{cases}\] and
\[   
\overline{f}(x,t) = 
\begin{cases}
f(x, t), &~\text{if}~t>\underline{u}_{\lambda}\\
f(x, \underline{u}_{\lambda}),&~\text{if}~t\leq \underline{u}_{\lambda}
\end{cases}\]
where, $\underline{u}_{\lambda}$ is the solution to \eqref{squasinna}. Let $G(s)=\int_{0}^{s}\overline{g}(t)dt$ and $F(x,s)=\int_{0}^{s}\overline{f}(x,t)dt$. 
We now define the energy functional $\Phi:X_0\to \mathbb{R}$ by
	\begin{equation*}
		\Phi(u)=\frac{1}{p}\|u\|^p-\lambda\int_{\Omega}G(u)dx-\int_{\Omega}F(x,u)dx.
	\end{equation*}
	Under the assumptions $(f_1)-(f_5)$ and $(g_1)-(g_2)$, the functional $\Phi$ is $C^1$ on $X_0$ (see Lemma 6.4 in \cite{saoudi2018multiplicity}) and weakly lower semicontinuous	by a standard arguments. Define
	\begin{align*}
		\zeta(u)&=\langle\Phi'(u),u\rangle_{X_0^*,X_0}=\|u\|^{p}-\lambda\int_{\Omega}g(u)udx-\int_{\Omega}f(x,u)udx, \forall u\in X_0,\\
		\mathcal{N}&\doteq\{u\in X_0\setminus\{0\}: \zeta(u)=0\},
	\end{align*}
	where $X_0^*$ is the dual space of $X_0$. For simplicity, we will denote $\langle\cdot, \cdot\rangle_{X_0^*,X_0}$ by $\langle\cdot, \cdot\rangle$. Clearly, every nontrivial solutions of (\ref{1.1}) belongs to $\mathcal{N}$.\\
	Define the set of sign-changing solutions of (\ref{1.1}) as
	\begin{equation*}\label{8-27-1}
		\mathcal{M}=\{u\in X_0: u^{\pm}\neq0, \langle\Phi'(u), u^+\rangle=\langle\Phi'(u),	u^-\rangle=0\},
	\end{equation*}
	where $u^+(x)\doteq\max\{u(x),0\}$, $u^-(x)\doteq\min\{u(x),0\}$. 
	We set $m_{\alpha}\doteq \inf\limits_{u\in \mathcal{M}}\Phi(u)$ and
	$c_{\alpha}\doteq \inf\limits_{u\in \mathcal{N}}\Phi(u)$.
	The main result proved in this article is the following.		
	\begin{theorem}\label{Thm1.1}
		Suppose that the assumptions $(f_1)-(f_5)$ and $(g_1)-(g_2)$ holds.
		Then there exists a $\Lambda>0$, such that for $\lambda\in(0, \Lambda)$, the problem (\ref{1.1}) admits one sign-changing solution $u^*\in
		X_0$ and $\Phi(u^*)=m_{\alpha}$.
	\end{theorem}
	\noindent The paper is organized as follows. In Section 2, we present some
	useful notations and give some preliminary results. In Section 3,
	we apply the method of Nehari manifold to
	prove Theorem \ref{Thm1.1}. Throughout the paper, we always denote
	by $C_1, C_2, \cdots$ positive constants (possibly different in
	different places) and let $|\cdot|_p$ denote the usual $L^p(\Omega)$
	norm for all $p\in[1,+\infty]$.

	\section{Important Lemmas}
	We begin this section with the following Lemma.
	\begin{lemma}\label{lambda finite}
		Assume $0<\delta<1<q < p_s^{*}-1$. Then $0<\varLambda<\infty$.
	\end{lemma}
	\begin{proof}
		This result can be proved by working on the similar lines as of \cite{saoudi2018multiplicity}.
	\end{proof}
	\noindent The following Lemma due to \cite{BL}, will be useful in the proof of Theorem \ref{Thm1.1}.
	\begin{lemma}\label{Ines}
		
		\begin{enumerate}[label=(\roman*)]
			\item For $2<p<\infty$, there exists $d_1,d_2>0$ such that, for all $\xi, \eta\in \mathbb{R}^N$,
			\begin{align}\label{7-30-21}
			\begin{split}
			&(|\xi|^{p-2}\xi-|\eta|^{p-2}\eta)(\xi-\eta)\geq d_1|\xi-\eta|^{p},\\ &||\xi|^{p-2}\xi-|\eta|^{p-2}\eta|\leq d_2 (|\xi|+|\eta|)^{p-2}|\xi-\eta|
			\end{split}
			\end{align}
			\item For $1<p\leq2$, there exist $d_3,d_4>0$ such that, for all $\xi, \eta\in \mathbb{R}^N$,
			\begin{align}\label{7-26-2}
			\begin{split}
			&(|\xi|^{p-2}\xi-|\eta|^{p-2}\eta)(\xi-\eta)\ge
			d_3\frac{|\xi-\eta|^2}{(|\xi|+|\eta|)^{2-p}},
			\\&||\xi|^{p-2}\xi-|\eta|^{p-2}\eta|\le
			d_4|\xi-\eta|^{p-1}
			\end{split}
			\end{align}
		\end{enumerate}
		\end{lemma}
	\noindent We have the following comparison principle for the fractional $p$-Laplacian operator.
	\begin{lemma}[Weak Comparison Principle]\label{weak comparison}
		Let $u, v\in X_0$. Suppose, $(-\Delta_p)^sv-\frac{\lambda}{v^{\gamma}}\geq(-\Delta_p)^su-\frac{\lambda}{u^{\gamma}}$ weakly with $v=u=0$ in $\mathbb{R}^N\setminus\Omega$.
		Then $v\geq u$ in $\mathbb{R}^N.$
	\end{lemma}
	\begin{proof}
		Since, $(-\Delta_p)^sv-\frac{\lambda}{v^{\gamma}}\geq(-\Delta_p)^su-\frac{\lambda}{u^{\gamma}}$ weakly with $u=v=0$ in $\mathbb{R}^N\setminus\Omega$, we have
		\begin{align}\label{compprinci}
			\langle(-\Delta_p)^sv,\phi\rangle-\int_{\Omega}\frac{\lambda\phi}{v^{\gamma}}dx&\geq\langle(-\Delta_p)^su,\phi\rangle-\int_{\Omega}\frac{\lambda\phi}{u^{\gamma}}dx, ~\forall{\phi\geq 0\in X_0}.
		\end{align}
		In particular choose $\phi=(u-v)^{+}$. To this choice, \eqref{compprinci} looks as follows.
		\begin{align}\label{compprinci1}
			\langle(-\Delta_p)^sv-(-\Delta_p)^su,(u-v)^{+}\rangle-\int_{\Omega}\lambda(u-v)^{+}\left(\frac{1}{v^{\gamma}}-\frac{1}{u^{\gamma}}\right)dx&\geq 0.
		\end{align}
 
		Let $\psi=u-v$. 
		The identity 
		\begin{align}\label{identity}
			|b|^{p-2}b-|a|^{p-2}a&=(p-1)(b-a)\int_0^1|a+t(b-a)|^{p-2}dt
		\end{align}
		with $a=v(x)-v(y)$, $b=u(x)-u(y)$ gives
		\begin{align}
			|u(x)-u(y)|^{p-2}(u(x)-u(y))-|v(x)-v(y)|^{p-2}(u(x)-u(y))\nonumber\\
			=(p-1)\{(u(y)-v(y))-(u(x)-v(x))\}Q(x,y)
		\end{align}
		where 
		\begin{align}
			Q(x,y)&=\int_0^1|(u(x)-u(y))+t((v(x)-v(y))-(u(x)-u(y)))|^{p-2}dt.
		\end{align}
		We choose the test function $\phi=(u-v)^{+}$. We express,
		\begin{align}
			\psi&=u-v=(u-v)^{+}-(u-v)^{-}\nonumber
		\end{align}
		to further obtain
		\begin{align}\label{negativity}
			[\psi(y)-\psi(x)][\phi(x)-\phi(y)]&=-(\psi^{+}(x)-\psi^{+}(y))^2.
		\end{align}
		The equation \eqref{negativity} implies
		\begin{align}
			0&\geq \langle(-\Delta_p)^sv-(-\Delta_p)^su,(v-u)^{+}\rangle\nonumber\\&=-(p-1)\frac{Q(x,y)}{|x-y|^{N+sp}}(\psi^{+}(x)-\psi^{+}(y))^2\nonumber\\&\geq0.
		\end{align}
		This leads to the conclusion that the Lebesgue measure of $\Omega^{+}$, i.e., $|\Omega^{+}|=0$. In other words $v\geq u$ a.e. in $\Omega$.
	\end{proof}

	\section{Proof of Theorem \ref{Thm1.1}}
	
	In this section we prove the existence of a sign-changing solution for \eqref{1.1} by obtaining a minimizer of the energy functional $\Phi$ over
	\begin{equation*}
	\mathcal{M}=\{u\in X_0: u^{\pm}\neq0, \langle\Phi'(u), u^+\rangle=\langle\Phi'(u),	u^-\rangle=0\}.
	\end{equation*}
	Further we will verify that the obtained minimizer is a sign-changing solution to \eqref{1.1}. Since, it is difficult to show that $\mathcal{M}\neq\emptyset$, we will prove that $\mathcal{M}\neq\emptyset$ by using the parametric method. We prove that, if $u\in X_0$ with $u^{\pm}\neq0$, the there exists a unique pair $(s, t)\in\mathbb{R}^+\times\mathbb{R}^+$, such that $su^+ + tu^-\in\mathcal{M}$. Finally to conclude that the minimizer of the constrained problem is a sign-changing solution, we use the quantitative deformation lemma (see Lemma 2.3 of \cite{Willem}) and Brouwer degree theory.
		
	\begin{lemma}\label{Lem-B1}
		Let the assumptions of Theorem \ref{Thm1.1} holds, then there exist
		$\mu_1,\mu_2>0$ and $\lambda_1>0$ such that
		\begin{description}
			\item[(i)]
			$\|u^{\pm}\|\ge \mu_1, \forall u \in \mathcal{M}$;
			\item[(ii)]$\int_{\Omega}|u^{\pm}|^qdx\ge \mu_2,\forall u \in
			\mathcal{M}$.
		\end{description}
	\end{lemma}
	\begin{proof}
		We have $\langle\Phi'(u),u^{\pm}\rangle=0$ for every $u\in \mathcal{M}$. Therefore,
		$$\lambda\int_\Omega g(u)u^{\pm} dx+\int_\Omega f(x, u)u^{\pm} dx = \lambda\int_\Omega g(u^{\pm})u^{\pm} dx+\int_\Omega f(x,u^{\pm})u^{\pm} dx.$$
		By a simple computation one can obtain
		\begin{equation*}
			\langle\Phi'(u),u^+\rangle=\langle\Phi'(u^+),u^+\rangle+2 C_1^+(u),
		\end{equation*}
		where,
		\begin{equation*}
			0<C_1^+(u)\doteq
			\int_{\Omega^+}\int_{\Omega^-}\frac{|u^+(x)-u^-(y)|^{p-1}u^+(x)}{|x-y|^{N+p\alpha
			}}dxdy-\int_{\Omega^+}\int_{\Omega^-}\frac{|u^+(x)|^{p}}{|x-y|^{N+p\alpha
			}}dxdy.
		\end{equation*}
		Therefore, $\langle\Phi'(u^{+}), u^{+}\rangle<0$, and hence it follows that
		\begin{equation*}
			\|u^{+}\|^p<\lambda\int_\Omega g(u^+)u^+ dx+\int_\Omega f(x, u^+)u^+ dx
		\end{equation*}
		Similarly, we obtain
		\begin{equation*}
			\|u^{-}\|^p<\lambda\int_\Omega g(u^-)u^- dx+\int_\Omega f(x, u^-)u^- dx
		\end{equation*}
		Now by the assumptions $(f_1)-(f_2)$, we have for every	$\epsilon>0$, there exists $C_\epsilon>0$ such that
		\begin{equation}\label{8-28-1}
		f(x,\tau)\tau\le \epsilon |\tau|^{p}+C_{\epsilon}|\tau|^{q}, \forall x\in\overline{\Omega},
		~~\forall \tau\in \mathbb{R}.
		\end{equation}
		Therefore, by the Sobolev inequality and the growth condition of $g$, there exists
		$C_1\,,C_2>0$ such that
		\begin{equation}\label{lambda1 small}
			\|u^{\pm}\|^{p}\le \epsilon C_1
			\|u^{\pm}\|^p+C_{\epsilon}C_1\|u^{\pm}\|^{q} + C_2\lambda\|u^{\pm}\|^{1-\delta}.
		\end{equation}
		We now choose, $\lambda>0$ (say, $\lambda_1$) very small such that
		\begin{equation}\label{lambda1}
		C_2\lambda_1\|u^{\pm}\|^{1-\delta}\leq\frac{1}{4}\|u^{\pm}\|^{p}
		\end{equation}
		Since, $q\in(p,p_{\alpha}^*)$, by \eqref{lambda1 small} and \eqref{lambda1} and for $\epsilon=\frac{1}{2C_1}$, one can see (i) holds. Again, by (\ref{8-28-1}), \eqref{lambda1 small} and \eqref{lambda1}, we have
		\begin{align*}
			\mu_1^p&\leq \|u^{\pm}\|^{p}\\
			&\leq \epsilon C_1
			\|u^{\pm}\|^p+C_{\epsilon}|u^{\pm}|_q^q + C_2\lambda\|u^{\pm}\|^{1-\delta}\\
			&\leq \epsilon C_1
			\|u^{\pm}\|^p+C_{\epsilon}|u^{\pm}|_q^q +\frac{1}{4}\|u^{\pm}\|^{p}
		\end{align*}
		Therefore, for $\epsilon=\frac{1}{2C_1}$, we can obtain that
		\begin{equation*}
			|u^{\pm}|_q^q\ge \frac{\mu_1^p}{4C_{\epsilon}}\doteq \mu_2.
		\end{equation*}
		This completes the proof.
	\end{proof}
		
	\begin{lemma}\label{Lem-unique}
		Let $u\in X_0$ be such that $u^{\pm}\neq0$. Then there exists a unique pair $(t_{u}, s_{u})\in\mathbb{R^+}\times\mathbb{R^+}$ such that $t_{u}u^{+} + s_{u}u^{-}\in \mathcal{M}$.
	\end{lemma}
	\begin{proof}
		For every $t, s>0$, let us define $g_1$ and $g_2$ as
		\begin{align*}
			g_1(t,s)&=\langle\Phi'(tu^++su^-), tu^+\rangle\\
			&=\int_{\Omega^+}\int_{\Omega^+}\cfrac{|tu^+(x)-tu^+(y)|^{p}}{|x-y|^{N+p\alpha
			}}dxdy+\int_{\Omega^+}\int_{\Omega^c}\cfrac{|tu^+(x)|^{p}}{|x-y|^{N+p\alpha
			}}dxdy\\
			&\hspace{1.0cm}+\int_{\Omega^c}\int_{\Omega^+}\cfrac{|tu^+(y)|^{p}}{|x-y|^{N+p\alpha
			}}dxdy
			+\int_{\Omega^+}\int_{\Omega^-}\cfrac{|tu^+(x)-su^-(y)|^{p-1}tu^+(x)}{|x-y|^{N+p\alpha }}dxdy\\
			&\hspace{1.0cm}+\int_{\Omega^-}\int_{\Omega^+}\cfrac{|su^-(x)-tu^+(y)|^{p-1}tu^+(y)}{|x-y|^{N+p\alpha }}dxdy-\lambda\int_\Omega g(tu^+)tu^+ dx
			-\int_{\Omega}f(x,tu^+)tu^+dx
		\end{align*}
		and
		\begin{align*}
			g_2(t, s)&=\langle\Phi'(tu^++su^-), su^-\rangle\\
			&=\int_{\Omega^+}\int_{\Omega^-}\frac{|tu^+(x)-su^-(y)|^{p-1}(-su^-(y))}{|x-y|^{N+p\alpha }}dxdy+\int_{\Omega^c}\int_{\Omega^-}\frac{|-su^-(y)|^{p}}{|x-y|^{N+p\alpha}}dxdy\\
			&\hspace{1.0cm}+\int_{\Omega^-}\int_{\Omega^+}\frac{|su^-(x)-tu^+(y)|^{p-1}(-su^-(x))}{|x-y|^{N+p\alpha}}dxdy+\int_{\Omega^-}\int_{\Omega^c}\frac{|su^-(x)|^{p}}{|x-y|^{N+p\alpha}}dxdy\\
			&\hspace{1.0cm}+\int_{\Omega^-}\int_{\Omega^-}\frac{|su^-(x)-su^-(y)|^{p}}{|x-y|^{N+p\alpha}}dxdy-\lambda\int_\Omega g(su^-)su^- dx-\int_{\Omega}f(x,su^-)su^-dx.
		\end{align*}
		Now by using $(f_4)$, we have for any $C_1>0$, there exists $C_2>0$ such that
		\begin{equation}\label{8-28-2}
		f(x, \tau)\tau\ge C_1 |\tau|^{p}-C_{2}, ~~\forall x\in \overline{\Omega},\forall \tau\in \mathbb{R}.
		\end{equation}
		Therefore, by using $q\in(p,p_{\alpha}^*)$, (\ref{8-28-1}), (\ref{8-28-2}) and
		Lemma \ref{Lem-B1}, there exist $r_1>0$, $\lambda>0$ small enough
		and $R_1>0$ large enough such that
		\begin{eqnarray}
		&&g_1(t,t)>0,~ g_2(t,t)>0,~~ \forall t\in (0,r_1),\label{8-28-3}\\
		&&g_1(t,t)<0,~ g_2(t,t)<0,~~ \forall t\in (R_1,+\infty).\label{8-28-3(2)}
		\end{eqnarray}
		Observe that, for a fixed $t>0$, $g_1(t,s)$ is increasing in $s$ on $(0,+\infty)$ and for a fixed $s>0$, $g_2(t,s)$ is increasing in $t$ on $(0,+\infty)$. Therefore, by using (\ref{8-28-3}) and \eqref{8-28-3(2)} there exist $\lambda>0$, $r>0$ and $R>0$ with $r<R$ such that
		\begin{eqnarray}
		&&g_1(r,s)>0,~ g_1(R,s)<0,~~ \forall s\in (r,R],\label{8-28-4}\\
		&&g_2(t,r)>0,~ g_2(t,R)<0,~~ \forall t\in (r,R].\label{8-28-4(2)}
		\end{eqnarray}
		Now, on applying the Miranda's theorem \cite{M}, $g_1(t_u,s_u)=g_2(t_u,s_u)=0$, for some $t_u,s_u\in[r,R]$. This implies that $t_uu^++s_uu^-\in \mathcal{M}$.\\		
		We now prove the uniqueness. Assume there exists $(t_1, s_1)$ and
		$(t_2,s_2)$ such that $t_iu^++s_iu^-\in \mathcal{M}$, $i=1,2$. We prove the uniqueness by dividing into two cases.\\
		{\bf Case 1.} Let $u\in \mathcal{M}$.\\
		Without loss of generality, we assume $(t_1,s_1)=(1,1)$ and $t_2\leq s_2$. Now, for $u\in X_0$, we define
		\begin{align*}
			A^+(u)=&\int_{\Omega^+}\int_{\Omega^+}\frac{|u^+(x)-u^+(y)|^{p}}{|x-y|^{N+p\alpha}}dxdy+\int_{\Omega^+}\int_{\Omega^c}\frac{|u^+(x)|^{p}}{|x-y|^{N+p\alpha}}dxdy\nonumber\\
			&+\int_{\Omega^c}\int_{\Omega^+}\frac{|u^+(y)|^{p}}{|x-y|^{N+p\alpha
			}}dxdy+\int_{\Omega^+}\int_{\Omega^-}\frac{|u^+(x)-u^-(y)|^{p-1}u^+(x)}{|x-y|^{N+p\alpha }}dxdy\nonumber\\
			&+\int_{\Omega^-}\int_{\Omega^+}\frac{|u^-(x)-u^+(y)|^{p-1}u^+(y)}{|x-y|^{N+p\alpha}}dxdy
		\end{align*}
		and
		\begin{align*}
			A^-(u)=&\int_{\Omega^-}\int_{\Omega^-}\frac{|u^-(x)-u^-(y)|^{p}}{|x-y|^{N+p\alpha}}dxdy+\int_{\Omega^-}\int_{\Omega^c}\frac{|u^-(x)|^{p}}{|x-y|^{N+p\alpha}}dxdy\nonumber\\
			&+\int_{\Omega^c}\int_{\Omega^-}\frac{|-u^-(y)|^{p}}{|x-y|^{N+p\alpha}}dxdy+\int_{\Omega^+}\int_{\Omega^-}\frac{|u^+(x)-u^-(y)|^{p-1}(-u^-(y))}{|x-y|^{N+p\alpha }}dxdy\nonumber\\
			&+\int_{\Omega^-}\int_{\Omega^+}\frac{|u^-(x)-u^+(y)|^{p-1}(-u^-(x))}{|x-y|^{N+p\alpha}}dxdy.
		\end{align*}
		Since, $u\in\mathcal{M}$, therefore, by using $\langle\Phi'(u),u^+\rangle=\langle\Phi'(u),u^-\rangle=0$, we get
		\begin{eqnarray}
		&&A^+(u)=\int_{\Omega}f(x,u^+)u^+dx+\lambda\int_{\Omega} g(u^+)u^+ dx,\label{8-29-1}\\
		&&A^-(u)=\int_{\Omega}f(x,u^-)u^-dx+\lambda\int_{\Omega} g(u^-)u^- dx. \label{8-29-2}
		\end{eqnarray}
		Again by using $<\Phi'(t_2u^++s_2u^-), t_2u^+>=0=<\Phi'(t_2u^++s_2u^-),
		s_2u^->$ we have
		\begin{align}
		&t_2^p(A^+(u)+B_1^+(u)+B_2^+(u))=\int_{\Omega}f(x,t_2u^+)t_2u^+dx
		+\lambda\int_{\Omega} g(t_2u^+)t_2u^+ dx\label{8-29-3}\\
		&s_2^p(A^-(u)+B_1^-(u)+B_2^-(u))=\int_{\Omega}f(x,s_2u^-)s_2u^-dx+\lambda\int_{\Omega} g(t_2u^+)t_2u^+ dx\label{8-29-4}
		\end{align}
		where,
		\begin{align*}
		&B_1^+(u)=\int_{\Omega^+}\int_{\Omega^-}\frac{|u^+(x)-\frac{s_2}{t_2}u^-(y)|^{p-1}u^+(x)}{|x-y|^{N+p\alpha}}dxdy-\int_{\Omega^+}\int_{\Omega^-}\frac{|u^+(x)-u^-(y)|^{p-1}u^+(x)}{|x-y|^{N+p\alpha}}dxdy,\\
		&B_2^+(u)=\int_{\Omega^-}\int_{\Omega^+}\frac{|\frac{s_2}{t_2}u^-(x)-u^+(y)|^{p-1}u^+(y)}{|x-y|^{N+p\alpha}}dxdy-\int_{\Omega^-}\int_{\Omega^+}\frac{|u^-(x)-u^+(y)|^{p-1}u^+(y)}{|x-y|^{N+p\alpha}}dxdy,\\
		&B_1^-(u)=\int_{\Omega^+}\int_{\Omega^-}\frac{|\frac{t_2}{s_2}u^+(x)-u^-(y)|^{p-1}(-u^-(y))}{|x-y|^{N+p\alpha}}dxdy-\int_{\Omega^+}\int_{\Omega^-}\frac{|u^+(x)-u^-(y)|^{p-1}(-u^-(y))}{|x-y|^{N+p\alpha}}dxdy,\\
		&B_2^-(u)=\int_{\Omega^-}\int_{\Omega^+}\frac{|u^-(x)-\frac{t_2}{s_2}u^+(y)|^{p-1}(-u^-(x))}{|x-y|^{N+p\alpha}}dxdy-\int_{\Omega^-}\int_{\Omega^+}\frac{|u^-(x)-u^+(y)|^{p-1}(-u^-(x))}{|x-y|^{N+p\alpha}}dxdy.
		\end{align*}
		Furthermore, $t_2\leq s_2$, implies that $B_1^+(u), B_2^+(u)\ge0$.
		Hence, from (\ref{8-29-1}) and (\ref{8-29-3}), we have
	{\small	\begin{align}\label{3.13}
		&\int_{\Omega}\left(\frac{f(x,t_2u^+)}{|t_2u^+|^{p-2}t_2u^+}-\frac{f(x,u^+)}{|u^+|^{p-2}u^+}\right)|u^+|^p+\lambda\int_{\Omega}\left(\frac{g(t_2u^+)}{|t_2u^+|^{p-2}t_2u^+}-\frac{g(u^+)}{|u^+|^{p-2}u^+}\right)|u^+|^p\nonumber\\&=\int_{\Omega}I_1+\int_{\Omega}I_2\nonumber\\
		&\geq0
		\end{align}}
	where, $I_1=\left(\frac{f(x,t_2u^+)}{|t_2u^+|^{p-2}t_2u^+}-\frac{f(x,u^+)}{|u^+|^{p-2}u^+}\right)|u^+|^p$ and $I_2=\lambda\left(\frac{g(t_2u^+)}{|t_2u^+|^{p-2}t_2u^+}-\frac{g(u^+)}{|u^+|^{p-2}u^+}\right)|u^+|^p$.
	{\bf Claim.} $t_2\geq 1$.\\
	To prove our claim, we consider the following four possibilities.
	\begin{enumerate}[label=\bf\Roman*.]
		\item When $I_1>0$, $I_2>0$: Now, $I_1>0$ implies that $t_2\geq 1$ by $(f_5)$. Again, $I_2>0$ implies $t_2\leq 1$ by using $(g_1)$-$(g_2)$. Therefore, on combining both the cases, we get $t_2=1$.
		\item When $I_1>0$, $I_2<0$: Since, $I_2<0$, we have $\int_{\Omega}I_1+\int_{\Omega}I_2\leq\int_{\Omega}I_1$. Therefore, by $(f_5)$, $I_1>0$ implies $t_2\geq 1$ and similarly, $I_2<0$ implies $t_2\geq 1$ by $(g_1)$-$(g_2)$. Thus $t_2\geq 1$.
		\item  When $I_1<0$, $I_2>0$: Since, $I_1<0$, we may choose, $\lambda>0$ small enough such that $\int_{\Omega}I_1+\int_{\Omega}I_2\leq 0$, which is a contradiction to \eqref{3.13}.
		\item  When $I_1<0$, $I_2<0$: In this case, both $I_1<0$ and $I_2<0$, yield $\int_{\Omega}I_1+\int_{\Omega}I_2\leq 0$, which is a contradiction to \eqref{3.13}.
	\end{enumerate}
		Therefore, we can conclude that $t_2\geq 1$. Again, from $B_1^-(u), B_2^-(u)\le0$, we have by (\ref{8-29-2}) and
		(\ref{8-29-4}) that
		\begin{equation*}
		\int_{\Omega}[\frac{f(x,s_2u^-)}{|s_2u^-|^{p-2}s_2u^-}-\frac{f(x,u^-)}{|u^-|^{p-2}u^-}]|u^-|^pdx+\lambda\int_{\Omega}[\frac{g(s_2u^-)}{|s_2u^-|^{p-2}s_2u^-}-\frac{g(u^-)}{|u^-|^{p-2}u^-}]|u^-|^pdx\leq 0,
		\end{equation*}
		on proceeding as the above proof together with $(f_5)$, $(g_1)$-$(g_2)$, one can prove that $s_2\le1$. Hence, $t_2=s_2=1$.\\
		{\bf Case 2.} Let $u\notin \mathcal{M}$.\\
		Let	$v_1=t_1u^++s_1u^-$ and $v_2=t_2u^++s_2u^-$. Again, by using above arguments, it is easy to prove that $\frac{t_2}{t_1}=\frac{s_2}{s_1}=1$. Hence, 	$(t_1, s_1)=(t_2, s_2)$. This completes the proof.		
	\end{proof}

	\begin{lemma}\label{Lem-minimizer}
		Assume $(f_1)$-$(f_5)$ and $(g_1)$-$(g_2)$ holds. Then there exists $u\in\mathcal{M}$ such that $\Phi(u)=m_\alpha$, where, $m_\alpha\doteq\inf\limits_{u\in \mathcal{M}}{\Phi(u)}$.
	\end{lemma}
	\begin{proof}
		Clearly, by the above Lemma \ref{Lem-unique}, we have $\mathcal{M}\neq\emptyset$.	Consider a minimizing sequence $\{u_n\}\subset \mathcal{M}$ such that $\Phi(u_n)\to m_\alpha$ as $n\to+\infty.$\\
		{\bf Claim:} The sequence $\{u_n\}$ is uniformly bounded in $X_0$.\\
		{\it Proof.} We will prove by contradiction. Let us assume that $\|u_n\|\rightarrow\infty$. We set $w_n=\frac{u_n}{\|u_n\|}$. Clearly, $\|z_n\|=1$, and upto a subsequence, there exists $w_0\in X_0$ such that
		\begin{enumerate}[label=(\roman*)]
			\item $w_n\rightarrow w_0$ in $X_0$,
			\item $w_n\rightarrow w_0$ in $L^r(\Omega)$ for all $r\in[1, p_{\alpha}^*)$ and
			\item  $w_n(x)\rightarrow w_0(x)$ almost everywhere in $\Omega$.
		\end{enumerate}
		 We further claim that $w_0=0$. Suppose not, define $\Omega_1=\{x\in\Omega:w_0(x)\neq 0\}$, then by $(f_4)$ and Fatou's lemma, we get,
		\begin{align*}
		\frac{1}{p}-\frac{m_{\alpha}+o(1)}{\|u_n\|^p}&=\frac{1}{p}-\frac{\Phi(u_n)}{\|u_n\|^p}\\
		&=\int_{\Omega}\frac{F(x, u_n)}{u_n^p}w_n^pdx+\int_{\Omega}\frac{G(u_n)}{u_n^p}w_n^pdx\\ &\geq\int_{\Omega_1}\frac{F(x, u_n)}{u_n^p}w_n^pdx+\int_{\Omega_1}\frac{G(u_n)}{u_n^p}w_n^pdx\\
		&\geq\int_{\Omega_1}\frac{F(x, u_n)}{u_n^p}w_n^pdx\rightarrow\infty~\text{as}~n\rightarrow\infty.
		\end{align*}
		which is a contradiction. Thus, $w_0\equiv0$. Therefore, $\{u_n\}$ is uniformly bounded in $X_0$. Then there exists $u^*\in X_0$ such that
		\begin{eqnarray}
		&&u_n^{\pm}\rightharpoonup (u^*)^{\pm}~~\mbox{in}~~ X_0,\label{8-30-1}\\
		&&u_n^{\pm}\to (u^*)^{\pm}~~\mbox{in} ~ L^r(\Omega)~~\mbox{for}~~r\in[1,p_{\alpha}^*),\label{8-30-2}\\
		&&u_n(x)\to u^*(x)~~ a.e.~~ x\in\Omega.\label{8-30-3}
		\end{eqnarray}
		From Lemma \ref{Lem-B1}, we have $(u^*)^{\pm}\neq0$. In addition, under the assumptions	$(f_1)$-$(f_2)$ and $(g_1)$-$(g_2)$, by using the compact embedding of $X_0\hookrightarrow L^r(\Omega)$ for $r\in[1,p_{\alpha}^*)$ and by applying some standard arguments (see \cite{Willem}), we get that
		\begin{align}
		&\lim\limits_{n\to+\infty}\int_{\Omega}f(x,u_n^{\pm})u_n^{\pm}dx=\int_{\Omega}f(x,(u^*)^{\pm})(u^*)^{\pm}dx,\label{8-30-4}\\
		&\lim\limits_{n\to+\infty}\int_{\Omega}F(x,u_n^{\pm})dx=\int_{\Omega}F(x,(u^*)^{\pm})dx\label{8-30-5}\\
		&\lim\limits_{n\to+\infty}\int_{\Omega}g(u_n^{\pm})u_n^{\pm}dx=\int_{\Omega}g((u^*)^{\pm})(u^*)^{\pm}dx,\label{8-30-4(1)}\\
		&\lim\limits_{n\to+\infty}\int_{\Omega}G(u_n^{\pm})dx=\int_{\Omega}G((u^*)^{\pm})dx\label{8-30-4(2)}.
		\end{align}
		 From Lemma \ref{Lem-unique}, we have the existence of $t^*,s^*>0$ such that $t^*(u^*)^++s^*(u^*)^-\in \mathcal{M}$. This implies
		\begin{align}
		(t^*)^p[A^+(u^*)+B_1^+(u^*)&+B_2^+(u^*)]\nonumber\\=&\int_{\Omega}f(x,t^*(u^*)^{+})t^*(u^*)^{+}dx+\int_{\Omega}g(t^*(u^*)^{+})t^*(u^*)^{+}dx,\label{8-30-6}\\
		(s^*)^p[A^-(u^*)+B_1^-(u^*)&+B_2^-(u^*)]\nonumber\\=&\int_{\Omega}f(x,s^*(u^*)^{-})s^*(u^*)^{-}dx+\int_{\Omega}g(s^*(u^*)^{-})s^*(u^*)^{-}dx.\label{8-30-7}
		\end{align}	
		We now prove that $t^*, s^*\leq 1$.	Since, the minimizing sequence $\{u_n\}\subset\mathcal{M}$, we get $\langle\Phi'(u_n), u_n^{\pm}\rangle=0$, which implies that
		\begin{eqnarray}\label{8-30-8}
		A^{\pm}(u_n)=\int_{\Omega}f(x,u_n^{\pm})u_n^{\pm}dx+\int_{\Omega}g(u_n^{\pm})u_n^{\pm}dx.
		\end{eqnarray}
		Therefore, by using the above inequalities (\ref{8-30-1})-(\ref{8-30-6}) and Fatou's lemma, we obtain
		\begin{eqnarray}\label{8-30-9}
		A^{\pm}(u^*)\le\int_{\Omega}f(x,(u^*)^{\pm})(u^*)^{\pm}dx+\int_{\Omega}g((u^*)^{\pm})(u^*)^{\pm}dx.
		\end{eqnarray}
		Furthermore, without loss of generality, assume $t^*\le s^*$. Again from (\ref{8-30-7}) and (\ref{8-30-9}) and the fact $B_1^-(u^*), B_2^-(u^*)\le0$, we get
		\begin{align}\label{8-30-10}
		0&\le\int_{\Omega}[\frac{f(x,(u^*)^-)}{|(u^*)^-|^{p-2}(u^*)^-}-\frac{f(x,s^*(u^*)^-)}{|s^*(u^*)^-|^{p-2}s^*(u^*)^-}]|(u^*)^-|^pdx\nonumber\\&+\lambda\int_{\Omega}[\frac{g((u^*)^-)}{|(u^*)^-|^{p-2}(u^*)^-}-\frac{g(s^*(u^*)^-)}{|s^*(u^*)^-|^{p-2}s^*(u^*)^-}]|u^+|^pdx.
		\end{align}
		Now proceeding on similar arguments as in the proof of Lemma \ref{Lem-unique} and using (\ref{8-30-10}), one can easily obtain $s^*\leq 1$. Therefore, we have	$0<t^*\le s^*\le1$.\\
		Let us define,  $\mathcal{H}(x,\tau)=f(x,\tau)\tau-pF(x,\tau)$ and $\mathcal{H}_1(\tau)= g(\tau)\tau-pG(\tau)$. Then, from $(f_5)$, we have $\mathcal{H}(x,\tau)$ is increasing with respect to $\tau$ on $(0,+\infty)$, decreasing on $(-\infty, 0)$ and $\mathcal{H}(x,\tau)\geq 0$. Again by $(g_1)$-$(g_2)$, we have $\mathcal{H}_1(\tau)\leq 0$ for $\tau\in\mathbb{R}\setminus\{0\}$. Therefore, by the definition of $\Phi$ and Fatou's lemma, we get	
		\begin{align*}
			m_{\alpha}&\le \Phi(t^*(u^*)^++s^*(u^*)^-)\\
			&=\Phi(t^*(u^*)^++s^*(u^*)^-)-\frac{1}{p}\langle\Phi'(t^*(u^*)^++s^*(u^*)^-), t^*(u^*)^++s^*(u^*)^-\rangle\\
			&=\frac{1}{p}\int_{\Omega}\mathcal{H}(x,t^*(u^*)^++s^*(u^*)^-)dx+\frac{1}{p}\int_{\Omega}\mathcal{H}_1(t^*(u^*)^+ +s^*(u^*)^-)dx\\	&\leq\frac{1}{p}\int_{\Omega}\mathcal{H}(x,t^*(u^*)^++s^*(u^*)^-)dx
			\end{align*}
			\begin{align*}
			&=\frac{1}{p}[\int_{\Omega^+}\mathcal{H}(x,t^*(u^*)^+)dx+\int_{\Omega^-}\mathcal{H}(x,s^*(u^*)^-)dx]\\
			&\le\frac{1}{p}[\int_{\Omega^+}\mathcal{H}(x,(u^*)^+)dx+\int_{\Omega^-}\mathcal{H}(x,(u^*)^-)dx]\\
			&\le\liminf\limits_{n\to+\infty}\frac{1}{p}\int_{\Omega}\mathcal{H}(x,u_n)dx\\
			&=\lim\limits_{n\to+\infty}[\Phi(u_n)-\frac{1}{p}\langle\Phi'(u_n),u_n\rangle]\\
			&=m_{\alpha}.
		\end{align*}
		Thus, we conclude $t^*=s^*=1$ and hence $\Phi(u^*)=m_{\alpha}$. This completes the proof.		
	\end{proof}
	
	\begin{lemma}\label{Lem-maximum}
		Let $u\in \mathcal{M}$. Then for every $t, s\geq 0$ with $(t, s)\neq(1,1)$, we have
		\begin{equation*}
			\Phi(u)>\Phi(tu^++su^-).
		\end{equation*}
	\end{lemma}
	\begin{proof}
		For each $u\in X_0$ such that $u^{\pm}\neq0$, let us define $I_u:[0,+\infty)\times
		[0,+\infty)\to \mathbb{R}$ as
		\begin{equation*}
		I_u(t,s)= \Phi(tu^++su^-), ~~\forall\, t,s\ge0.
		\end{equation*}
		Observe that from $(f_4)$, we get
		\begin{equation*}
			\lim\limits_{|(t,s)|\to+\infty}I_u(t,s)=-\infty.
		\end{equation*}
		Therefore, $I_u$ admits a global maximum at some $(t_0,s_0)\in [0,+\infty)\times [0,+\infty)$. We now prove that $t_0>0, s_0>0$ by showing that the other three possibilities can not hold, which are as follow.
		\begin{description}
			\item[(i)]$t_0=s_0=0$;
			\item[(ii)]$t_0>0,s_0=0$;
			\item[(iii)]$t_0=0,s_0>0$.
		\end{description}
		Let $s_0=0$. Since, $I_u$ has a global maximum at $(t_0,s_0)$, then $\Phi(t_0u^+)\ge \Phi(tu^+)$ for every $t>0$,. Therefore, we have
		$\langle\Phi'(t_0u^+),t_0u^+\rangle=0$, which implies,
		\begin{equation}\label{8-30-11}
		t_0^p\|u^+\|^p=\int_{\Omega}f(x,t_0u^+)t_0u^+dx+\int_{\Omega}g(t_0u^+)t_0u^+dx.
		\end{equation}
		Again, since $u\in \mathcal{M}$, we get	$\langle\Phi'(u^+),u^+\rangle<0$, i.e.
		\begin{equation*}
			\|u^+\|^p<\int_{\Omega}f(x,u^+)u^+dx+\int_{\Omega}g(u^+)u^+dx.
		\end{equation*}
		Now, using this inequality and (\ref{8-30-11}), we get
		\begin{equation*}
			\int_{\Omega}\left[\frac{f(x,u^+)}{|u^+|^{p-2}u^+}-\frac{f(x,t_0u^+)}{|t_0u^+|^{p-2}t_0u^+}\right]|u^+|^pdx+\int_{\Omega}\left[\frac{g(u^+)}{|u^+|^{p-2}u^+}-\frac{g(t_0u^+)}{|t_0u^+|^{p-2}t_0u^+}\right]|u^+|^pdx>0.
		\end{equation*}
		Again, on repeating similar arguments as in Lemma \ref{Lem-unique}, together with $(g_1)$-$(g_2)$ and $(f_5)$, we get $t_0\leq 1$. Furthermore, $\mathcal{H}(x,\tau)\geq0, \forall (x,\tau)\in\overline{\Omega}\times\mathbb{R}$ and $\mathcal{H}_1(\tau)\leq0, \forall\,\tau\in\mathbb{R}\setminus\{0\}$. In addition, $\mathcal{H}(x,\tau)$ is increasing on	$(0,+\infty)$ and decreasing on $(-\infty, 0)$ with respect to $\tau$. Therefore, we have
		\begin{align*}
			I_u(t_0,0)&=\Phi(t_0u^+)\\
			&=\Phi(t_0u^+)-\frac{1}{p}\langle\Phi'(t_0u^+), t_0u^+\rangle\\
			&=\frac{1}{p}\int_{\Omega}\mathcal{H}(x,t_0u^+)dx+\frac{1}{p}\int_{\Omega}\mathcal{H}_1(t_0u^+)dx\\
			&\leq\frac{1}{p}\int_{\Omega}\mathcal{H}(x,t_0u^+)dx\\
			&\leq\frac{1}{p}\int_{\Omega^+}\mathcal{H}(x,u^+)dx\\
			&<\frac{1}{p}\left[\int_{\Omega^+}\mathcal{H}(x,u^+)dx+\int_{\Omega^-}\mathcal{H}(x,u^-)dx\right]\\
			&=\Phi(u)-\frac{1}{p}\langle\Phi'(u),u\rangle\\
			&=\Phi(u)=I_u(1,1).
		\end{align*}
		This contradicts that $I_u$ has a global maximum at $(t_0,s_0)$. Hence, $s_0>0$. Similarly, we can prove that $t_0>0$. Finally, Lemma \ref{Lem-unique}, guarantees that $(1,1)$ is the unique critical point of $I_u$ in	$(0,+\infty)\times(0,+\infty)$. This readily implies that, if $t_0,s_0\in(0,1]$ such that $(t_0,s_0)\neq (1,1)$, then we have
		\begin{equation*}
			I_u(t_0,s_0)<I_u(1,1).
		\end{equation*}
		This completes the proof.
	\end{proof}
		\noindent The following lemma concludes the existence of a critical point of $\Phi$, which is a least energy solution to our problem.	
	\begin{lemma}\label{Lem-critical-point}
		Let there exists $u^*\in \mathcal{M}$ such that $\Phi(u^*)=m_{\alpha}$. Then $u^*$ is a critical point of $\Phi$, i.e.,
		$\Phi'(u^*)=0$.
	\end{lemma}
	\begin{proof}
		We will prove by method of contradiction. Let $\Phi'(u^*)\neq0$, then there exist	$\rho_1, \mu_1>0$ such that
		\begin{equation*}
			\|\Phi'(u)\|\ge \rho_1, ~~\forall\, B_{3\mu_1}(u^*),
		\end{equation*}
		where $B_{3\mu_1}(u^*)=\{u\in X_0: \|u-u^*\|\le 3\mu_1\}$ a closed ball of radius $3\mu$ in $X_0$ centered at $u^*$. Now, $u^*\in\mathcal{M}$ implies that $(u^*)^{\pm}\neq0$, then we can choose a sufficiently small $\mu_1>0$ such that $u^{\pm}\neq0$ for all $u\in B_{3\mu_1}(u^*)$. For sufficiently small $\delta_1\in(0,\frac{1}{2})$, let us define, $D=(1-\delta_1, 1+\delta_1)\times(1-\delta_1, 1+\delta_1)$ such that $t(u^*)^++s(u^*)^-\in B_{3\mu_1}(u^*)$ for all $(t,s)\in\overline{D}$. From Lemma \ref{Lem-maximum}, one can say
		\begin{equation}\label{8-31-5}
		\tilde{m}_{\alpha}\doteq \max\limits_{(t,s)\in \partial
			D}\Phi(t(u^*)^++s(u^*)^-)<m_{\alpha}.
		\end{equation}
		Choose, $\epsilon_1\doteq \min\{\frac{m_{\alpha}-\tilde{m}_{\alpha}}{2},
		\frac{\rho_1\mu_1}{8}\}$. Therefore, from the quantitative deformation lemma (see Lemma 2.3 of \cite{Willem}) it follows that there exists a continuous map $\eta: \mathbb{R}\times X_0\to X_0$ such that
		\begin{description}
			\item[(i)]$\eta(1,u)=u$ if $u\not\in\Phi^{-1}[m_{\alpha}-2\epsilon_1, m_{\alpha}+2\epsilon_1]\cap B_{2\mu_1}(u^*)$;
			\item[(ii)]$\eta(1,\Phi^{m_{\alpha}+\epsilon_1}\cap B_{\mu_1}(u^*))\subset
			\Phi^{m_{\alpha}-\epsilon_1}$;
			\item[(iii)]$\Phi(\eta(1,u))\le \Phi(u), \forall u\in X_0$.
		\end{description}
		We define, $\sigma(t,s)\doteq\eta(1,t(u^*)^++s(u^*)^-), \forall\, (t,s)\in\overline{D}$. Thus from the Lemma \ref{Lem-maximum} together with (ii)-(iii) of the deformation lemma, we get
		\begin{equation}\label{8-30-12}
		\max\limits_{(t,s)\in\overline{D}}\Phi(\eta(1,t(u^*)^++s(u^*)^-))
		<m_{\alpha},
		\end{equation}
		which implies that $\{\sigma(t,s)\}_{(t,s)\in\overline{D}}\cap
		\mathcal{M}=\emptyset$. Again, we will prove by the following argument that $\{\sigma(t,s)\}_{(t,s)\in\overline{D}}\cap\mathcal{M}\neq \emptyset$ to arrive at a contradiction. Now, for $(t,s)\in \overline{D}$, we define
		\begin{align*}
			&J_1(t,s)=(\langle\Phi'(t(u^*)^++s(u^*)^-), (u^*)^+\rangle,
			\langle\Phi'(t(u^*)^++s(u^*)^-), (u^*)^-\rangle),\\
			&J_2(t,s)=(\frac{1}{t}\langle\Phi'(\sigma(t,s)), \sigma^+(t,s)\rangle, \frac{1}{s}\langle\Phi'(\sigma(t,s)), \sigma^-(t,s)\rangle).
		\end{align*}
		Since $f, g\in C^1$, the functional $J_1$ is $C^1$. Therefore, from $\langle\Phi'(u^*), (u^*)^{\pm}\rangle=0$, we get
		\begin{align*}
			&\int_{Q}\frac{|u^*(x)-u^*(y)|^{p-2}(u^*(x)-u^*(y))((u^*)^+(x)-(u^*)^+(y))}{|x-y|^{N+p\alpha}}dxdy\\&\hspace{7cm}=\int_{\Omega}f(x,(u^*)^+)(u^*)^+dx+\int_{\Omega}g((u^*)^+)(u^*)^+dx
		\end{align*}
		
		\begin{align*}
			&\int_{Q}\frac{|u^*(x)-u^*(y)|^{p-2}(u^*(x)-u^*(y))((u^*)^-(x)-(u^*)^-(y))}{|x-y|^{N+p\alpha}}dxdy=\\&\hspace{7cm}\int_{\Omega}f(x,(u^*)^-)(u^*)^-dx+\int_{\Omega}g((u^*)^-)(u^*)^-dx.~~~~~~~~~
		\end{align*}
		From $(g_1)$-$(g_2)$ and $(f_5)$, we have
		$\mathcal{H}'(x,\tau)\tau=f'(x,\tau)\tau^2-(p-1)f(x,\tau)\tau>0$ for all $\tau\in\mathbb{R}\setminus\{0\}$. We denote
		\begin{align*}
			&\alpha_1=\int_{Q}\frac{|u^*(x)-u^*(y)|^{p-2}|(u^*)^+(x)-(u^*)^+(y)|^2}{|x-y|^{N+p\alpha}}dxdy,\\
			&\alpha_2=\int_{\Omega}f'_u(x,(u^*)^+)|(u^*)^+|^2dx+\int_{\Omega}g'_u((u^*)^+)|(u^*)^+|^2dx,\\& \alpha_3=\int_{\Omega}f(x,(u^*)^+)(u^*)^+dx+\int_{\Omega}g((u^*)^+)(u^*)^+dx\\
			&\beta_1=\int_{Q}\frac{|u^*(x)-u^*(y)|^{p-2}|(u^*)^-(x)-(u^*)^-(y)|^2}{|x-y|^{N+p\alpha}}dxdy,\\
			&\beta_2=\int_{\Omega}f'_u(x,(u^*)^-)|(u^*)^-|^2dx+\int_{\Omega}g'_u((u^*)^-)|(u^*)^-|^2dx,\\& \beta_3=\int_{\Omega}f(x,(u^*)^-)(u^*)^-dx+\int_{\Omega}g((u^*)^-)(u^*)^-dx,\\
			&\gamma_1=\int_{Q}\frac{|u^*(x)-u^*(y)|^{p-2}((u^*)^-(x)-(u^*)^-(y))((u^*)^+(x)-(u^*)^+(y))}{|x-y|^{N+p\alpha}}dxdy,\\
			&\gamma_2=\int_{Q}\frac{|u^*(x)-u^*(y)|^{p-2}((u^*)^+(x)-(u^*)^+(y))((u^*)^-(x)-(u^*)^-(y))}{|x-y|^{N+p\alpha}}dxdy.
		\end{align*}
		It is easy to observe that
		\begin{align*}
			&\alpha_1>0,~~ \alpha_2>(p-1)\alpha_3>0,\\
			&\beta_1>0,~~ \beta_2>(p-1)\beta_3>0,\\
			&\gamma_1=\int_{Q}\frac{|u^*(x)-u^*(y)|^{p-2}(-(u^*)^-(x)(u^*)^+(y)-(u^*)^-(y)(u^*)^+(x))}{|x-y|^{N+p\alpha}}dxdy=\gamma_2>0,\\
			&\alpha_1+\gamma_1=\alpha_3,~~ \beta_1+\gamma_2=\beta_3.
		\end{align*}
		Hence, we get
		\begin{align*}
			&\det(J'_1(1,1))\\
			&=\langle\Phi''(u^*)(u^*)^+, (u^*)^+\rangle\cdot
			\langle\Phi''(u^*)(u^*)^-, (u^*)^-\rangle\\
			&-\langle\Phi''(u^*)(u^*)^+,
			(u^*)^-\rangle\cdot \langle\Phi''(u^*)(u^*)^-,  (u^*)^+\rangle\\
			&=[(p-1)\alpha_1-\alpha_2]\cdot [(p-1)\beta_1-\beta_2]-(p-1)^2\gamma_1\cdot \gamma_2\\
			&>(p-1)^2\gamma_1\cdot \gamma_2-(p-1)^2\gamma_1\cdot \gamma_2=0.
		\end{align*}
		Therefore, by using the Brouwer degree theory, we get $\deg(J_1,D,0)=1$. Again, from (\ref{8-30-12}), we have $\sigma(t,s)=t(u^*)^++s(u^*)^-,	\forall\, (t,s)\in\partial D$. Hence,
		\begin{equation*}
			\deg(J_2,D,0)=\deg(J_1,D,0)=1.
		\end{equation*}
		Therefore, there exists $(t_0,s_0)\in D$ such that $J_2(t_0,s_0)=0$. On using the conditions (i)-(ii) in the deformation lemma, one can obtain that
		$$u_0\doteq\sigma(t_0,s_0)=\eta(1,t_0(u^*)^++s_0(u^*)^-)\in
		B_{3\mu_1}(u^*).$$ Therefore, we can say $\langle\Phi'(u_0), u_0^+\rangle=\langle\Phi'(u_0), u_0^-\rangle=0$ such that $u_0^{\pm}\neq0$, that is, $u_0\in \{\eta(t,s)\}_{(t,s)\in\overline{D}}\cap\mathcal{M}$.  Hence, we have a contradiction. Thus we conclude that $u^*$ is a critical point of $\Phi$ and a least energy sign-changing solution of problem corresponding to $\Phi$. Finally, since the critical points of $\Phi$ are also critical points of $I_{\lambda}$, we have $u^*$ is a critical point of $I_{\lambda}$. Hence $u^*$ is a sign-changing solution to the problem (\ref{1.1}).		
	\end{proof}	
	
	\section*{Acknowledgement}
	The author S. Ghosh, thanks the Council of Scientific and Industrial Research (C.S.I.R), India, for the financial assistantship received to carry out this research work.


\begin{thebibliography}{99}
		\bibitem{Ad}
		Adams R. A.,
		\newblock Sobolev spaces,
		\newblock {\em Pure and Appl. Math.}, 65, Academic Press, New York, 1975.
		
		\bibitem{AS}
		Alves C. O. and Souto M. A.,
		\newblock Existence of least energy nodal solution for a Schr\"odinger-Poisson system in bounded domains,
		\newblock {\em Z. Angew. Math. Phys.}, 65, 1153-1166, 2014.
		
		\bibitem{BL}
		Bartsch T. and Liu Z. L.,
		\newblock On a superlinear elliptic $p$-Laplacian equation,
		\newblock {\em J. Differential Equations}, 198, 149-175, 2004.
		
		\bibitem{BW}
		Bartsch T. and Weth T.,
		\newblock Three nodal solutions of singularly perturbed elliptic equations on domains without topology,
		\newblock {\em Ann. Inst. H. Poincar\'e Anal. Non Lin\'eaire}, 22, 259-281, 2005.
		
		\bibitem{BSRT}
		Bonheure D., Santos E., Ramos M. and Tavares H.,
		\newblock Existence and symmetry of least energy nodal solutions for Hamiltonian elliptic systems,
		\newblock {\em J. Math. Pures Appl.}, 104, 1075-1107, 2015.
		
		\bibitem{BCde}
		Br\"andle C., Colorado E., Pablo A. and S\'anchez U.,
		\newblock A	concave-convex elliptic problem involving the fractional Laplacian,
		\newblock {\em Proc. Roy. Soc. Edinburgh Sect. A}, 143, 39-71, 2013.
		
		\bibitem{CaSi}
		Caffarelli L. and Silvestre L.,
		\newblock An extension problem related to the fractional Laplacian,
		\newblock {\em Comm. Partial Differential Equations}, 32, 1245-1260, 2007.
		
		\bibitem{canino2017nonlocal}
		Canino A., Montoro L., Sciunzi B. and Squassina M.,
		\newblock Nonlocal problems with singular nonlinearity,
		\newblock {\em Bulletin des Sciences Math\'{e}matiques}, 141(3), 223-250, 2017.
		
		\bibitem{CW}
		Chang X. J. and Wang Z. Q.,
		\newblock Nodal and multiple solutions of nonlinear problems involving the fractional Laplacian,
		\newblock {\em J. Differential Equations}, 256, 2965-2992, 2004.
		
		\bibitem{chang}
		Chang X., Nie Z. and Wang Z.,
		\newblock Sign-Changing Solutions of Fractional $p$-Laplacian Problems,
		\newblock {\em Advanced Nonlinear Studies}, 19(1), 29-53, 2019.
		
		\bibitem{crandall1977dirichlet}
		Crandall M. G., Rabinowitz P. H. and Tartar L.,
		\newblock On a Dirichlet problem with a singular nonlinearity,
		\newblock {\em Communications in Partial Differential Equations},
		2(2), 193-222, 1977.
		
		\bibitem{dhanya2012global}
		Dhanya R., Giacomoni J., Prashanth S. and Saoudi K.,
		\newblock Global bifurcation and local multiplicity results for elliptic
		equations with singular nonlinearity of super exponential growth in
		$\mathbb{R}^2$,
		\newblock {\em Advances in Differential Equations}, 17(3/4), 369-400, 2012.
		
		\bibitem{DPV}
		Di Nezza E., Palatucci G. and Valdinoci E.,
		\newblock Hitchhiker's guide to the fractional Sobolev spaces,
		\newblock {\em Bull. Sci. Math.}, 136, 521-573, 2012.
		
		\bibitem{GTZ}
		Gao Z., Tang X. H. and Zhang W.,
		\newblock Least energy sign-changing solutions for nonlinear problems involving fractional Laplacian,
		\newblock {\em Electron. J. Differential Equations}, 238, 10 pp, 2016.
		
		\bibitem{ghanmi2016multiplicity}
		Ghanmi A. and Saoudi K.,
		\newblock A multiplicity results for a singular problem involving the
		fractional $p$-Laplacian operator,
		\newblock {\em Complex variables and elliptic equations}, 61(9), 1199-1216,
		2016.

		\bibitem{giacomoni2009multiplicity}
		Giacomoni J. and Saoudi K.,
		\newblock Multiplicity of positive solutions for a singular and critical
		problem,
		\newblock {\em Nonlinear Analysis: Theory, Methods \& Applications},
		71(9), 4060-4077, 2009.

		\bibitem{giacomoni2007multiplicity}
		Giacomoni J. and Sreenadh K.,
		\newblock Multiplicity results for a singular and quasilinear equation,
		\newblock {\em Discrete and Continuous Dynamical Systems}, 2007(special), 429-435, 2007.
		
		\bibitem{haitao2003multiplicity}
		Haitao Y.,
		\newblock Multiplicity and asymptotic behavior of positive solutions for a
		singular semilinear elliptic problem,
		\newblock {\em Journal of Differential Equations}, 189(2), 487-512, 2003.
		
		\bibitem{hirano2004existence}
		Hirano N., Saccon C. and Shioji N.,
		\newblock Existence of multiple positive solutions for singular elliptic
		problems with concave and convex nonlinearities,
		\newblock {\em Advances in Differential Equations}, 9(1-2), 197-220, 2004.
		
		\bibitem{LLW2}
		Liu J. Q., Liu X. Q. and Wang Z. Q.,
		\newblock Sign-changing solutions for coupled nonlinear Schr\"odinger equations with critical growth,
		\newblock {\em J. Differential Equations}, 261, 7194-7236, 2016.
		
		\bibitem{M}
		Miranda C.,
		\newblock Un'osservazione su un teorema di Brouwer,
		\newblock {\em Bol. Un. Mat. Ital.}, 3, 5-7, 1940.
		
		\bibitem{mukherjee2016dirichlet}
		Mukherjee T. and Sreenadh K.,
		\newblock On Dirichlet problem for fractional $p$-Laplacian with singular
		non-linearity,
		\newblock {\em Advances in Nonlinear Analysis}, 2016.
		
		\bibitem{Ra}
		Rabinowitz P.,
		\newblock Minimax methods in critical point theory with applications to differential equations,
		\newblock {\em CBMS Regional Conference Series in Mathematics}, 65, American Mathematical Society, Providence, 1986.
		
		\bibitem{saoudi2017critical}
		Saoudi K.,
		\newblock A critical fractional elliptic equation with singular nonlinearities,
		\newblock {\em Fractional Calculus and Applied Analysis}, 20(6), 1507-1530,	2017.
		
		\bibitem{saoudi2018multiplicity}
		Saoudi K., Ghosh S. and Choudhuri D.,
		\newblock Multiplicity and H{\"o}lder regularity of solutions for a nonlocal elliptic PDE involving singularity,
		\newblock {\em arXiv preprint arXiv:1808.02469}, 2018.

		\bibitem{servadei2012mountain}
		Servadei R. and Valdinoci E.,
		\newblock Mountain pass solutions for non-local elliptic operators,
		\newblock {\em Journal of Mathematical Analysis and Applications}, 389(2), 887-898, 2012.
		
		\bibitem{servadei2013variational}
		Servadei R. and Valdinoci E.,
		\newblock Variational methods for non-local operators of elliptic type,
		\newblock {\em Discrete and Continuous Dynamical Systems}, 33(5), 2105-2137, 2013.
		
		\bibitem{Va}
		Valdinoci E.,
		\newblock From the long jump random walk to the fractional Laplacian,
		\newblock {\em Bol. Soc. Esp. Mat. Apl. SMA}, 49, 33-44, 2009.
		
		\bibitem{WZ1}
		Wang Z. P. and Zhou H. S.,
		\newblock Sign-changing solutions for the nonlinear Schr\"odinger-Poisson system in $\mathbb{R}^3$,
		\newblock {\em Calc. Var. Partial Differential Equations}, 52, 927-943, 2015.
		
		\bibitem{WZ2}
		Wang Z. P. and Zhou H. S.,
		\newblock Radial sign-changing solution for fractional Schr\"odinger equation,
		\newblock {\em Discrete Contin. Dyn. Syst.}, 36, 499-508, 2016.
		
		\bibitem{Wei}
		Wei S.,
		\newblock Sign-changing solutions for a class of Kirchhoff-type problem in bounded domains,
		\newblock {\em J. Differential Equations}, 259, 1256-1274, 2015.
		
		\bibitem{Willem}
		Willem M.,
		\newblock Minimax theorems,
		\newblock {\em Progress in Nonlinear Differential Equations and their Applications}, Vol. 24, \rm Birkh\"auser, Boston, 1996.		
		
	\end{thebibliography}
\end{document}